\documentclass{elsart3-1}

 \usepackage{graphicx}

\usepackage{amssymb,amsmath}

\usepackage{caption}
\usepackage{subfigure}

\usepackage[english,francais]{babel}

\newtheorem{theorem}{Theorem}[section]

\newtheorem{e-proposition}[theorem]{Proposition}
\newtheorem{corollary}[theorem]{Corollary}
\newtheorem{e-definition}[theorem]{Definition\rm}
\newtheorem{remark}{\it Remark\/}
\newtheorem{example}{\it Example\/}

\renewcommand{\theequation}{\arabic{equation}}
\def\og{\leavevmode\raise.3ex\hbox{$\scriptscriptstyle\langle\!\langle$~}}
\def\fg{\leavevmode\raise.3ex\hbox{~$\!\scriptscriptstyle\,\rangle\!\rangle$}}

\newenvironment{proof}
   {{\bf Proof:}}
   {\qed}

\newcommand{\conv}{\operatorname{Conv}}

\usepackage{color}

\journal{the Acad\'emie des sciences}
\begin{document}
\centerline{}
\begin{frontmatter}


\selectlanguage{english}
\title{Bounding Stability Constants for Affinely Parameter-Dependent Operators}


\selectlanguage{english}
\author[authorlabel1]{Robert O'Connor}
\ead{oconnor@aices.rwth-aachen.de}

\address[authorlabel1]{RWTH Aachen University, Aachen, Germany}


\medskip
\begin{center}
{\small Received *****; accepted after revision +++++\\
Presented by £££££}
\end{center}

\begin{abstract}
\selectlanguage{english}
In this article we introduce new possibilities of bounding the stability constants that play a vital role in the reduced basis method.  By bounding stability constants over a neighborhood we make it possible to guarantee stability at more than a finite number of points and to do that in the offline stage. We additionally show that Lyapunov stability of dynamical systems can be handled in the same framework.
{\it To cite this article: R. O'Connor, C. R. Acad. Sci. Paris, Submitted (2016).}

\vskip 0.5\baselineskip

\selectlanguage{francais}
\noindent{\bf R\'esum\'e} \vskip 0.5\baselineskip \noindent
{\bf Des bornes inf\'erieures pour les constantes de stabilit\'e associ\'es \`a des operateurs avec une d\'ependance affine des param\'etres. }
Nous pr\'esentons des nouvelles m\'ethodes pour borner les constantes de stabilit\'e qui jouent un r\^ole essentiel dans les approximations par bases r\'eduites.  Notres m\'ethodes nous permettent de borner les constant dans toute une voisinage et non seulement \`a une numero fini de points.  Nous montrons aussi qu'on peut d\'emontrer la stabilit\'e de Liapounov dans le m\^eme cadre.
{\it Pour citer cet article~: R. O'Connor, C. R. Acad. Sci.
Paris, Soumis (2016).}

\end{abstract}
\end{frontmatter}

\selectlanguage{francais}

\selectlanguage{english}
\section{Introduction}
\label{}

In the reduced basis method, stability constants play two important roles:  They ensure the numerical stability of the problem and they are a critical part of error bounds. Unfortunately in such contexts it is not possible to calculate stability constants for each new parameter value. Instead, lower bounds need to be used.  For the computation of such bounds many methods have been developed \cite{CHM+2009,HRS+2007,Nguyen2005,Veroy2003,VPR+2003}. An important characteristic of these methods is the offline-online decomposition of the workload. The offline stage, which is performed beforehand, is generally very expensive, but the online cost to approximate the stability constant for each new parameter value should be cheap.

Early efforts to bound stability constants \cite{Nguyen2005,VPR+2003} often made use of local information to bound the constants in small regions. By doing that in many small regions it is possible to bound the stability constants everywhere.  Other methods made use of more global information to bound stability constants \cite{CHM+2009,HRS+2007,Veroy2003}.  These methods require the solution of a linear programming problem for each new parameter point and are not suited for bounding stability constants at more than discrete points.  One such method is the successive constraints method (SCM) \cite{CHM+2009,HRS+2007}. SCM has proven to be very efficient when a posteriori error bounds are being calculated, but it is not sufficient when bounds are needed for the entire parameter domain.  That can be the case in real-time applications \cite{OG2016}, where stability and error tolerances need to be ensured beforehand: improving the model in real-time is not possible. 

The main contribution of this article is to show how information can be used more efficiently on a local scale. Our methods can be combined with SCM to bound stability constants over the entire parameter domain in an efficient manner. The same methods can also be used to prove that a system is Lyapunov stable. Compared to normal SCM our method has two advantages: it can reduce the online computational cost of bounding stability constants and, more importantly, it allows us to bound stability constants everywhere in the parameter domain. The disadvantage is that the offline stage can be much more costly than that of SCM.

\section{Problem Statement}\label{sec:state}

Let $\mathcal D\subset\mathbb R^p$ be a bounded parameter domain and let $X$ and $Y$ be Hilbert spaces.  In practice the spaces will often be finite dimensional but the theory that we will present also holds for infinite-dimensional spaces. We consider a parameter-dependent bilinear operator $a(\cdot,\cdot;\mu):X\times Y\rightarrow \mathbb R$ for $\mu\in \mathcal D$ with the following affine decomposition $a(v,w;\mu)=\sum_{q=1}^Q\Theta_q(\mu) a_q(v,w)$.  The affine decomposition separates the operator into parameter-dependent functions $\Theta_q(\cdot):\mathcal D\rightarrow \mathbb R$ and parameter-independent bilinear forms $a_q(\cdot,\cdot):X\times Y\rightarrow \mathbb R$. The efficiency of reduced-basis methods is largely a result of such decompositions \cite{RHP2008,VPR+2003}.  The bilinear operator can be associated with two different stability constants.

\begin{e-definition}\label{def:coercivity}
For a parameter-dependent bilinear operator $a(\cdot,\cdot;\mu)$ we define the inf-sup constant $\beta(\mu)$, and if $X=Y$, we also define the coercivity constant $\alpha(\mu)$.
\begin{equation}\label{eq:coercivity}
\beta(\mu):=\inf_{w\in X}\sup_{v\in Y}\frac{a(w,v;\mu)}{\|w\|_X\|v\|_Y},\hspace{1cm}\alpha(\mu):=\inf_{v\in X}\frac{a(v,v;\mu)}{\|v\|_X^2}
\end{equation}
For a given parameter value $\mu$ we will say that an operator is inf-sup stable (resp. coercive) if $\beta(\mu)>0$ ($\alpha(\mu)>0$).
\end{e-definition}

In this article we consider the problem of finding lower bounds that are valid over the entire parameter domain.  In particular, we will consider two types of problems that have received little attention in this context: (i) proving stability and (ii) computing sharp lower bounds for the stability constants.  Whenever inf-sup constants are needed, they can be reformulated using Riesz representations. The resulting problems can then be handled in much the same way as problems involving coercivity constants \cite{HRS+2007}. We can thus restrict our discussion to coercivity constants and assume that $X=Y$.

We present the first method that is well adapted to the simpler problem of proving stability. For the more complicated problem of estimating stability constants, earlier methods \cite{Nguyen2005,VPR+2003} exist, but we present a significantly more efficient one. In the next section we will review a result from Veroy \cite{Veroy2003} and show how it can be used to locally bound coercivity constants in a more accurate manner.

\section{Simplified Parameter Dependence}\label{sec:simp}

In order to better take advantage of the affine nature of the operator $a(\cdot,\cdot;\mu)$ we will define a simpler bilinear operator $a_\Theta(v,w;\Theta(\mu)):=a(v,w;\mu)$ for all $v,w\in X$ and $\mu\in\mathcal D$.  Here $\Theta(\cdot):\mathcal D\rightarrow \mathbb R^Q$ is defined such that $\Theta(\mu):=[\Theta_1(\mu),\Theta_2(\mu),\dots,\Theta_Q(\mu)]^T$. With the operator $a_\Theta(\cdot,\cdot;\psi)$ we will associate the following affine decomposition and coercivity constant
\begin{equation}\label{eq:theta_affine}
a_\Theta(v,w;\psi)=\sum_{q=1}^Q\psi_q a_q(v,w),\hspace{1cm}\alpha_\Theta(\psi):=\inf_{v\in X}\frac{a_\Theta(v,v;\psi)}{\|v\|_X^2}
\end{equation}
for any $\psi=[\psi_1,\dots,\psi_Q]^T\in\mathbb R^Q$. The following result, which shows the concavity of $\alpha_\Theta(\mu)$, was also proved by Veroy \cite{Veroy2003} but we provide a much simpler proof.

\begin{theorem}\label{thm:concavity}
Let $a_\Theta(\cdot,\cdot;\psi)$ be an operator with an affine parameter dependence of the form given in (\ref{eq:theta_affine}).  The coercivity constant $\alpha_\Theta(\psi)$ associated with $a_\Theta(\cdot,\cdot;\psi)$ is a concave function of $\psi\in\mathbb R^Q$.
\end{theorem}

\begin{proof}
We begin by defining the set $\mathcal Y:=\{y\in\mathbb R^Q|y_q=a_q(v,v)/\|v\|_X^2, \forall 1\leq q\leq Q$ and some $v\in X\}$, where $y_q$ is the $q$\textsuperscript{th} element of $y\in\mathbb R^Q$.  We can then write the coercivity constant $\alpha_\Theta(\psi)$ as the solution to the minimization problem $\alpha_\Theta(\psi)=\inf\{\psi^Ty|y\in\mathcal Y\}$ \cite{HRS+2007}.  For any $\eta,\rho\in\mathbb R^Q$ and $\tau\in[0,1]$ it holds that $\alpha_\Theta(\tau\eta+(1-\tau)\rho)=\inf_{y\in\mathcal Y}(\tau\eta+(1-\tau)\rho)^Ty\geq\tau\left(\inf_{y\in\mathcal Y}\eta^Ty\right)+(1-\tau)\left(\inf_{y\in\mathcal Y}\rho^Ty\right)=\tau\alpha_\Theta(\eta)+(1-\tau)\alpha_\Theta(\rho)$, which is the definition of concavity for $\alpha_\Theta(\psi)$.
\end{proof}

If we are interested in proving the stability of the operator $a_\Theta(\cdot,\cdot;\psi)$ over a given set of parameters, we can use the following corollary of theorem \ref{thm:concavity}.

\begin{corollary}\label{cor:convex_hull}
Assume that $a_\Theta(\cdot,\cdot;\psi)$ is an operator of the form given in (\ref{eq:theta_affine}).  For any set $\Psi$ of points in $\mathbb R^Q$ it holds that $\min\{\alpha_\Theta(\psi)|\psi\in\conv(\Psi)\}=\min\{\alpha_\Theta(\eta)|\eta\in\Psi\}$, where $\conv(\Psi)$ denotes the convex hull of $\Psi$.
\end{corollary}

Sharper bounds can be built using interpolation on simplexes.

\begin{corollary}\label{cor:interp}
Let $\Psi=\{\eta^i|1\leq i \leq m\leq Q+1\}\subset\mathbb R^Q$ be a set of $m$ points such that the dimension of $\conv(\Psi)$ is exactly $m-1$.  For all $\psi\in\conv(\Psi)$, unique interpolation coefficients $c_i(\psi)\in[0,1]$ are defined such that $\psi=\sum_{q=1}^mc_q(\psi)\eta^q$ and $1=\sum_{q=1}^mc_q(\psi)$. It then holds that $\alpha_\Theta(\psi)\geq\sum_{q=1}^mc_q(\psi)\alpha_\Theta(\eta^q)$ for all $\psi\in\conv(\Psi)$.
\end{corollary}

Given the set of points $\Psi$ from corollary \ref{cor:interp} it is also possible to extrapolate the values of $\alpha_\Theta(\psi)$.  This can be done to derive upper bounds for $\alpha_\Theta(\psi)$ over certain parts of $\mathbb R^Q$. Deriving upper bounds in this way is convenient because it requires only the information that is already needed for the lower bounds.

In some situations $\alpha_\Theta(\psi)$ is affine over a one-dimensional interval.  Understanding such situations can be useful in constructing and understanding bounds.  In particular, this phenomenon explains some of our numerical results.

\begin{theorem}\label{thm:linear}
Assume that $\|v\|_X^2=a_\Theta(v,v;\bar\psi)$ for some $\bar\psi\in\mathbb R^Q$ and that $a_\Theta(\cdot,\cdot;\psi)$ has the form given in (\ref{eq:theta_affine}), then $\alpha_\Theta(\bar\psi+\tau\rho)=1+\tau\alpha_\Theta(\rho)$ for all $\tau>0$ and $\rho\in\mathbb R^Q$.
\end{theorem}
\begin{proof}
The proof is straight forward using the definition of $\alpha_\Theta(\cdot)$ and the linearity of $a_\Theta(\cdot,\cdot;\psi)$ in $\psi$.
\end{proof}

\section{Bounding Coercivity Constants}\label{sec:bound}

Noting that $\alpha(\mu)=\alpha_\Theta(\Theta(\mu))$ for all $\mu\in\mathcal D$ we can use the results from the last section to derive bounds for $\alpha(\mu)$. Let us consider a small example of our method.
\begin{example}\label{exam}
Let $\mathcal D=[0,1]$, and consider the operators $a(v,w):=a_0(v,w)+\mu a_1(v,w)+\mu^2 a_2(v,w)$ and $a_\Theta(v,w;\psi):=\psi_0a_0(v,w)+\psi_1 a_1(v,w)+\psi_2 a_2(v,w)$.  We define the points $\eta^1=[1,0,0]^T$, $\eta^2=[1,1,0]^T$, and $\eta^3=[1,1,1]^T$ and the set $\Psi=\{\eta^1,\eta^2,\eta^3\}$ such that $\Theta(\mathcal D)=\{[1,\mu,\mu^2]^T|0\leq \mu \leq 1\}\subset\conv(\Psi)$. From corollary \ref{cor:interp} we know that $\alpha_\Theta(\psi)\geq(1-\psi_1)\alpha_\Theta(\eta^1)+(\psi_1-\psi_2)\alpha_\Theta(\eta^2)+\psi_2\alpha_\Theta(\eta^3)$ for all $\psi\in\conv(\Psi)$.  The equivalent result in $\mathcal D$ is given by $\alpha(\mu)\geq(1-\mu)\alpha_\Theta(\eta^1)+(\mu-\mu^2)\alpha_\Theta(\eta^2)+\mu^2\alpha_\Theta(\eta^3)$.
\end{example}

Earlier methods to locally bound stability constants build concave lower bounds for $\alpha(\mu)$ \cite{Nguyen2005,VPR+2003}.  Unfortunately, those bounds can be quite pessimistic. Taking advantage of the natural concavity of the operator in $\mathbb R^Q$, as we propose, should produce sharper bounds and make our algorithms more efficient. We will now show how these ideas can be used with larger parameter domains.

\subsection{Ensuring Stability}

We begin with the problem of ensuring the coercivity of $a(\cdot,\cdot;\mu)$ for all $\mu\in\mathcal D$. Due to the concave nature of $\alpha_\Theta(\psi)$ it suffices to prove coercivity for all $\psi\in\Gamma$, where $\Gamma$ denotes the boundary of $\Theta(\mathcal D)$. We start by choosing a finite set of points $\Psi\subset\mathbb R^Q$ such that $\Theta(\mathcal D)\subset \conv(\Psi)$ and calculating $\alpha_\Theta(\eta)$ for all $\eta\in\Psi$.  Corollary \ref{cor:convex_hull} tells us that we are done if $\alpha_\Theta(\psi)>0$ for all $\psi\in\Psi$. If that is not the case, we can make use of the more powerful corollary \ref{cor:interp}. We construct a set of non-overlapping simplexes that cover $\Gamma$ and have vertices in $\Psi$.  We can then interpolate bounds onto $\Gamma$.  The point $\psi\in\Gamma$ where the bound is the smallest is then added to $\Psi$ and we compute the coercivity constant there.  The simplexes are then refined and the bounds improved.  The process continues until either coercivity is proven for all $\psi\in\Gamma$, or a point $\psi\in\Gamma$ is found such that $\alpha_\Theta(\psi)\leq0$, which proves that the problem is not stable.

\subsection{Calculating Sharp Lower Bounds}

Another problem that we can handle is that of computing sharp lower bounds everywhere in $\mathcal D$. To do that we cover $\Theta(\mathcal D)\in\mathbb R^Q$ with a mesh of non-overlapping simplexes, calculate the value of $\alpha_\Theta(\psi)$ at each vertex $\psi$, and build interpolated lower bounds using corollary \ref{cor:interp}.

This basic idea can be made more efficient in two ways: we build the simplex mesh adaptively and use SCM. Rather than working with a predetermined simplex mesh it will usually be more efficient to build the mesh adaptively.  That reduces the number of vertices that are needed to achieve a certain tolerance. The idea of such an adaptive methods is to refine the mesh where the approximation error is too large.  In this case we compute lower and upper bounds for the coercivity in a simplex and use the difference to measure the accuracy.  Wherever the difference exceeds a certain tolerance we refine the mesh. To reduce the computational cost associated with each vertex we can use SCM to bound the stability constants at each vertex.

Veroy \cite{Veroy2003} made use of the same concavity but used it in a very different way.  Her method uses overlapping simplexes, and the one that produces the best bound is found using a linear program.  As is the case with SCM, the use of linear programming means that the method can only be used for discrete parameter values. In comparison our method, combined with SCM, produces much sharper bounds while allowing us to bound the constants everywhere.

The offline cost of our method will be significantly higher than that of using only SCM. In particular our method can become very expensive if $p$ is not reasonably small.  If in addition $Q>p$, as in example \ref{exam}, our method also has the disadvantage that we are working in a space with higher dimensionality than the parameter domain $\mathcal D$.

An advantage of our method is that it reduces the computational cost of the online stage.  Computing a lower bound with SCM requires searching for parameter points from a predetermined list, constructing a linear programming problem and solving it.  Using our method the online cost is just that of identifying the associated simplex and either interpolating (with corollary \ref{cor:interp}) or choosing the minimum value (with corollary \ref{cor:convex_hull}).  We note, however, that this method only affects the cost of evaluating the stability constant.  In the context of the reduced basis method it is generally also necessary to compute the reduced solution and the residual in the online stage.  Those computations may dominate the online computational cost. Avoiding those computations in the online stage requires the offline evaluation of error bounds. That is also necessary for real-time applications \cite{OG2016}.

\section{Lyapunov Stability}
In previous sections and in the field of reduced-basis modeling, ``stability'' has meant numerical stability, but in the area of control, which is often cited as an application area, there is a great interest in stability in the sense of Lyapunov.  In this section we review Lyapunov stability theory for linear systems and show how we can prove stability for parameter-dependent systems.  Here we assume that $X=Y$ is finite dimensional.

Let us consider a dynamical system of the form:
\begin{equation}\label{eq:system}
\left\langle \dot y(t),v\right\rangle_V=-a\left(y(t),v;\mu\right),\hspace{.25cm} \forall v\in X \text{ and }\forall t\geq0.
\end{equation}
Here $\langle \cdot,\cdot\rangle_V$ denotes an inner product on $X$ and will usually be associated with the mass matrix.  The state of the system is given by $y\in C^1([0,T];X)$ and its time derivative is denoted $\dot y$.

For a symmetric operator $a(\cdot,\cdot;\mu)$ the value of $\alpha(\mu)$ determines if the system is Lyapunov stable.

\begin{theorem}\label{thm:symmetric}
Let us assume that $a(\cdot,\cdot;\mu)$ is a symmetric operator.  The system (\ref{eq:system}) is Lyapunov stable for the parameter $\mu$ iff $\alpha(\mu)\geq 0$.  It is asymptotically stable if $\alpha(\mu)> 0$.
\end{theorem}

For nonsymmetric operators the situation is more complicated.  It is very common to use eigenvalues to classify stable systems, but that is not practical in our context:  there is no way of rigorously bounding the eigenvalues of nonsymmetric operators in an offline/online manner. Instead, we will make use of Lyapunov functions.  The following theorem gives a classical result and a connection to a new coercivity constant.

\begin{theorem}\label{thm:lyapunov}
For a fixed parameter $\mu$ the system in (\ref{eq:system}) is stable in the sense of Lyapunov iff there exists a symmetric, coercive bilinear operator $p(\cdot,\cdot)$ such that $\phi(v,w;\mu):=p(T_\mu v,w)+p(v,T_\mu w)$ is also coercive.  Here $T_\mu$ is the supremizing operator defined by $\langle T_\mu w,v\rangle_V=a(w,v;\mu)$ for all $w,v\in X$.
\end{theorem}

\begin{remark}
If the operator $a(\cdot,\cdot;\mu)$ is coercive for a particular parameter $\mu\in\mathcal D$, we can choose $p(v,w)=\langle v,w\rangle_V$.  We then get $\phi(v,w;\mu)=\langle T_\mu v,w\rangle_V+\langle v,T_\mu w\rangle_V=a(v,w;\mu)+a(w,v;\mu)$.  Theorem \ref{thm:lyapunov} and the assumption that $a(\cdot,\cdot;\mu)$ is coercive tell us that the system is Lyapunov stable.
\end{remark}

We will consider a fixed $p(\cdot,\cdot)$ and investigate the stability of the system for a range of parameter values. By showing that $\phi(\cdot,\cdot;\mu)$ is coercive we can prove that (\ref{eq:system}) is Lyapunov stable. Although the operator $p(\cdot,\cdot)$ is independent of $\mu$, the parameter dependence of $a(\cdot,\cdot;\mu)$ will induce an affine parameter dependence in $\phi(\cdot,\cdot;\mu)$. That allows us to use the theory from previous sections to ensure stability.  An effective method to construct the operator $p(\cdot,\cdot)$ was introduced by O'Connor \cite{OConnor2016}.

\section{Numerical Example: Diffusion-Convection-Reaction Equation}

We will consider a one-dimensional spatial domain $\Omega=(0,1)$ and a spatial operator given by $A(\mu):=-\Delta +\mu_1(x-0.5)\nabla+\mu_2I$ with a homogenous Dirichlet boundary at $x=0$ and a homogenous Neumann boundary at $x=1$. The operator $a(\cdot,\cdot;\mu)$ will be a discretization of the operator $A(\mu)$ written as a bilinear form. For the parameter domain we choose $\mathcal D:=[0,30]\times[-0.4,2]\subset\mathbb R^2$.  We will consider an equidistant piecewise-linear finite-elements discretization of $\Omega$ with 180 degrees of freedom.   For the norm on the finite-element space $X$ we choose $\|v\|_X^2=a(v,v;\bar\mu)$, with $\bar\mu=[0,0]^T$.

The first scenario that we consider is $\mu_2=0$. For $\mu_1\geq0$, $\alpha(\mu)$ is affine in $\mu_1$, with $\alpha([0,0]^T)=1$ and $\alpha([12.0908,0]^T)=0$.  We next define two symmetric bilinear operators $p^1(\cdot,\cdot)$ and $p^2(\cdot,\cdot)$ such that
\begin{equation}
p^1(T_{[20,0]^T} v,w)+p^1(v,T_{[20,0]^T} w)=2a(v,w;\bar\mu)=p^2(T_{[28.25,0]^T} v,w)+p^2(v,T_{[28.25,0]^T} w)
\end{equation}
for all $v,w\in X$. This is motivated by the fact that $a(\cdot,\cdot;\bar\mu)$ was used to define our norm $\|\cdot\|_X$. We can then define $\phi_1(\cdot,\cdot;\mu)$ and $\phi_2(\cdot,\cdot;\mu)$ and the associated coercivity constants $\alpha_{\phi1}(\mu)$ and $\alpha_{\phi2}(\mu)$ based on $p^1(\cdot,\cdot)$ and $p^2(\cdot,\cdot)$. This process is  described in more detail by O'Connor \cite{OConnor2016}. The coercivity constants $\alpha(\mu)$, $\alpha_{\phi1}(\mu)$, and $\alpha_{\phi2}(\mu)$ are shown in figure \ref{pics:a} for $\mu_2=0$. Since for all $\mu_1\in[0,30]$ at least one of the coercivity constants is positive we can conclude that the system is Lyapunov stable for $\mu_2=0$ and all $\mu_1\in[0,30]$.  Figure \ref{pics:b} shows that the system is also Lyapunov stable for the second scenario $\mu_2=2$. An important difference between figures \ref{pics:a} and \ref{pics:b} is that for $\mu_2=0$ we can apply theorem \ref{thm:linear}, which explains why figure \ref{pics:a} contains only straight line segments.  Considering both figures \ref{pics:a} and \ref{pics:b}  and making use of corollary \ref{cor:convex_hull} we can conclude that $a(\cdot,\cdot;\mu)$ is coercive for all $\mu\in\conv(\{[0,0]^T,[0,2]^T,[12,0]^T,[17,2]^T\})$.  Similarly, we can also prove the coercivity of $\phi_1(\cdot,\cdot;\mu)$ and $\phi_2(\cdot,\cdot;\mu)$ over certain convex sets. In that way we can prove the Lyapunov stability of the system for all $\mu\in[0,30]\times[0,2]$. Figure \ref{pics:c} shows the coercivity constants for the third scenario $\mu_2=-0.4$.  In that case $\alpha(\mu)$, $\alpha_{\phi1}(\mu)$, and $\alpha_{\phi2}(\mu)$ are not sufficient to prove Lyapunov stability for all $\mu_1\in[0,30]$.

\begin{figure}
 \centering
 \subfigure[]
 {
  \includegraphics[width=4.7cm,trim=1cm 0cm 1cm 0cm]{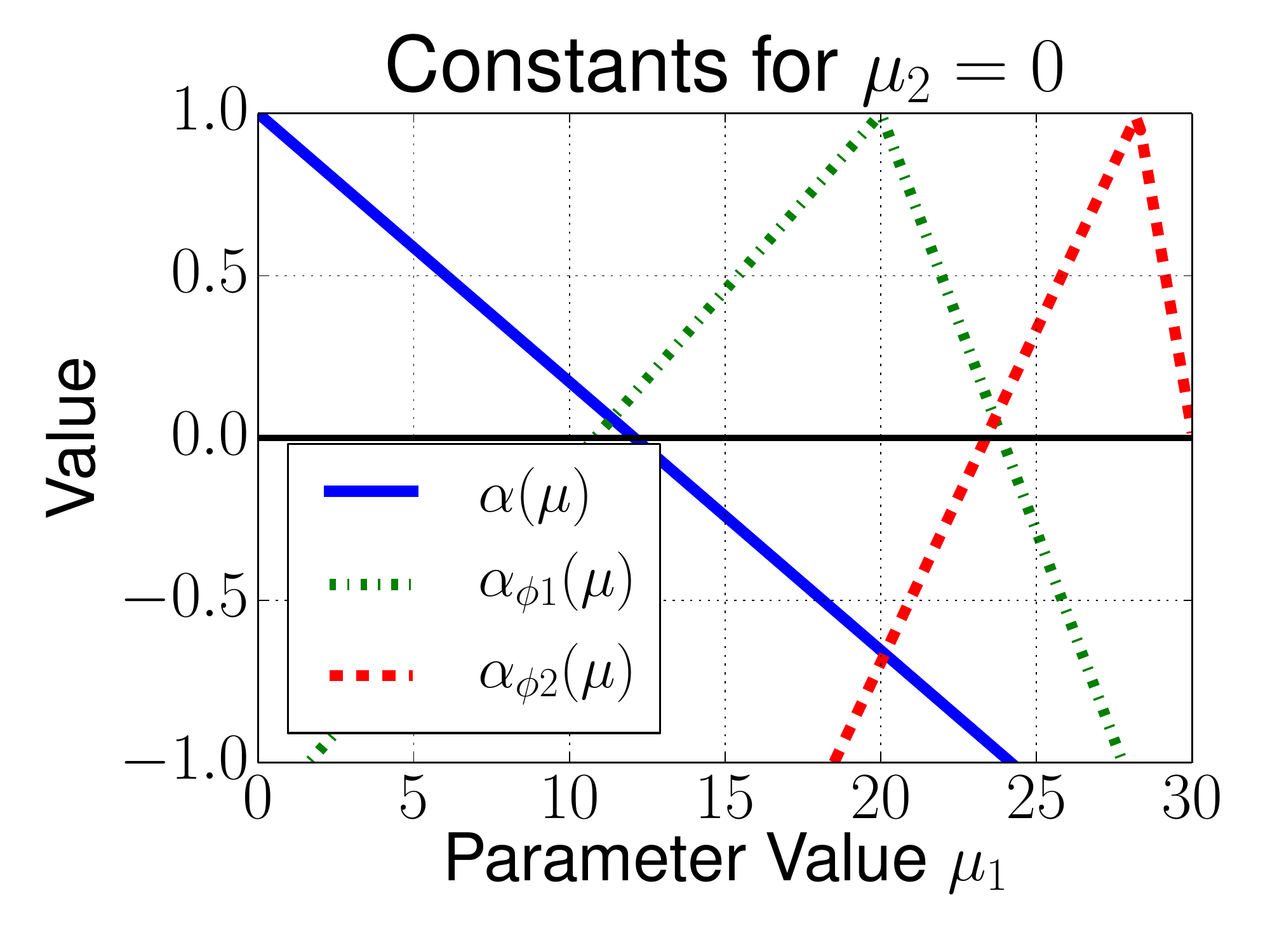}
   \label{pics:a}
   }
 \subfigure[]
 {
  \includegraphics[width=4.7cm,trim=1cm 0cm 1cm 0cm]{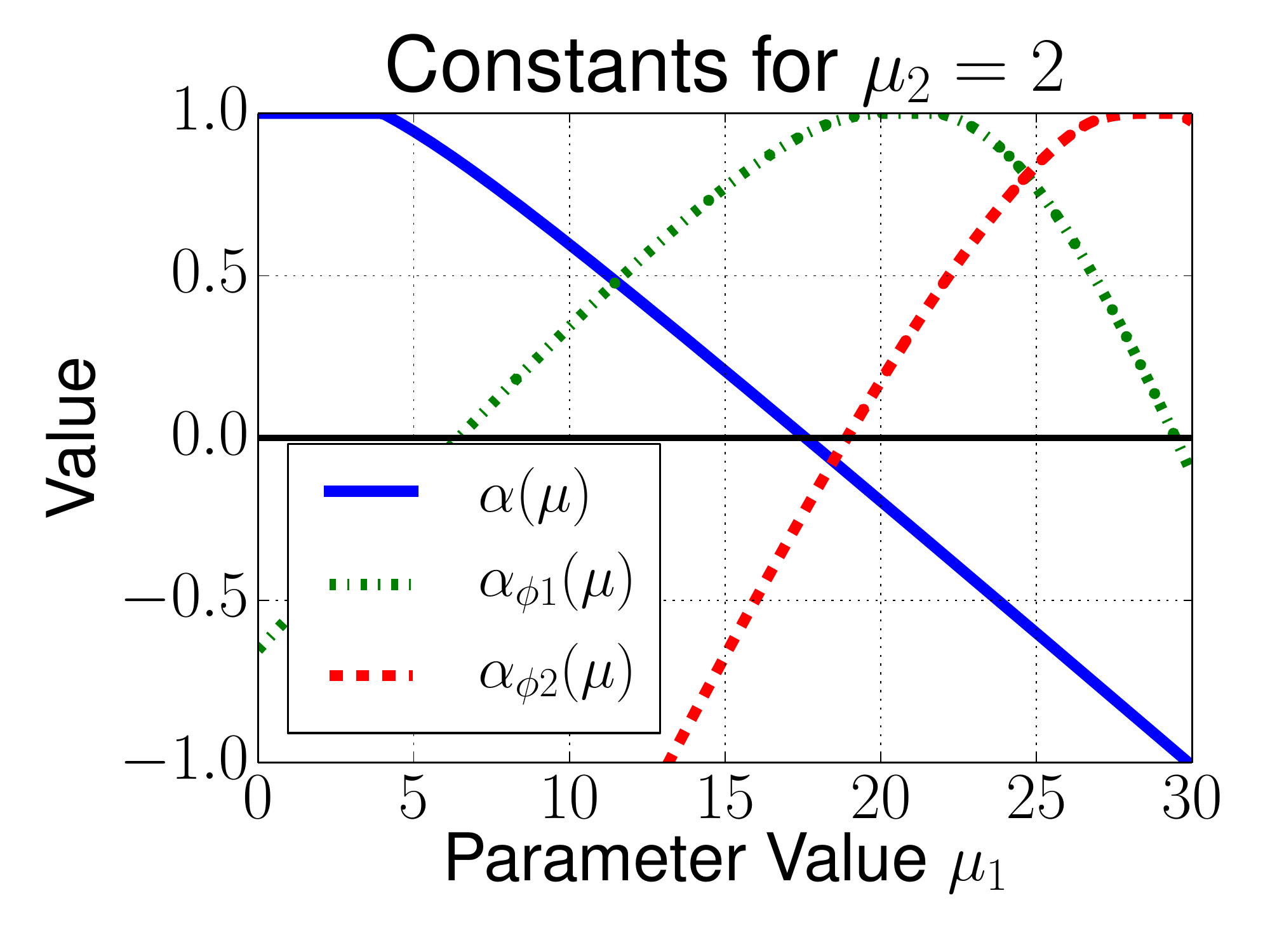}
   \label{pics:b}
   }
   \subfigure[]
 {
  \includegraphics[width=4.7cm,trim=1cm 0cm 1cm 0cm]{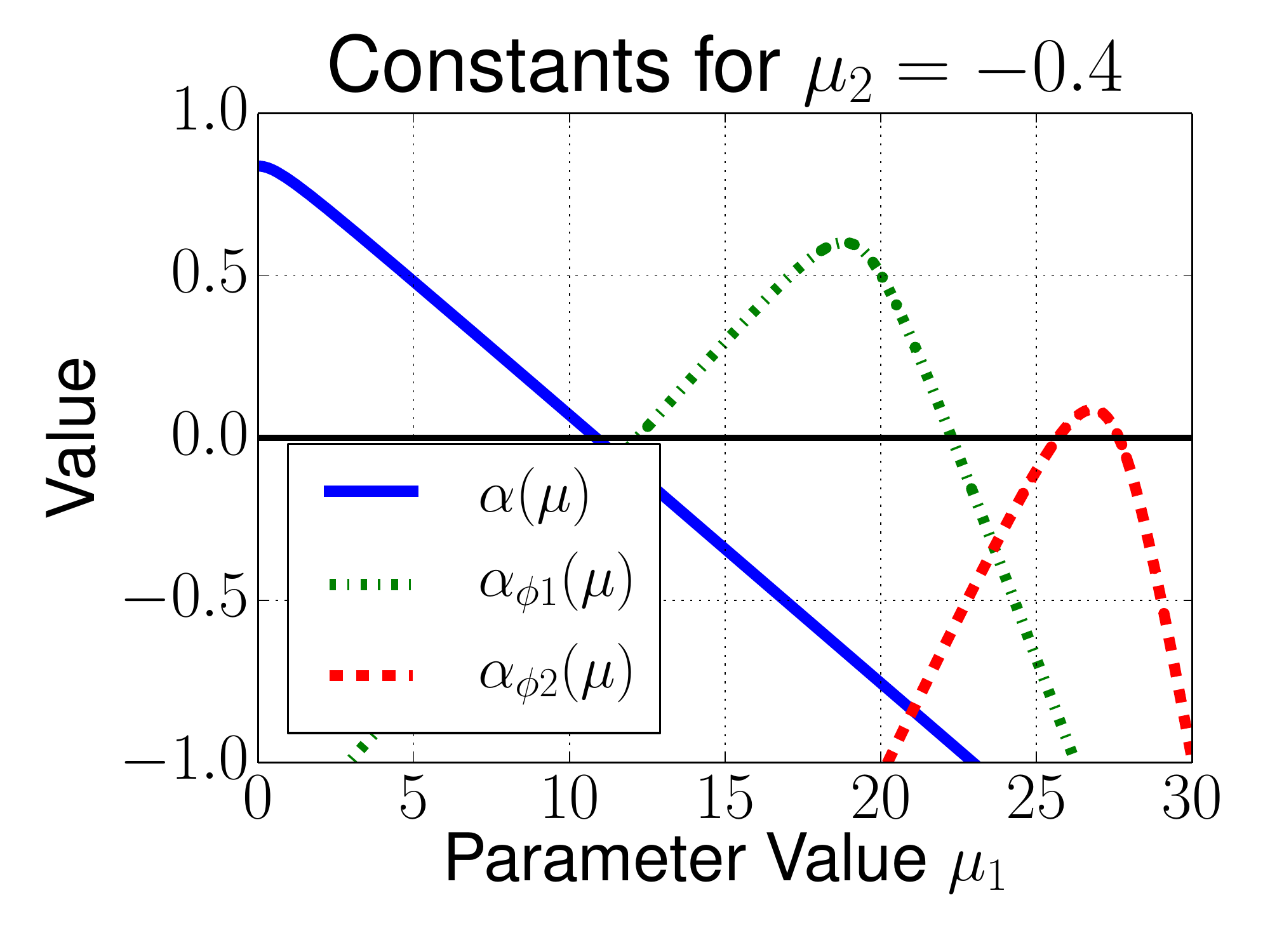}
   \label{pics:c}
   }
\vspace{-.3cm}

 \label{pics}
 \caption{Coercivity constants as a function of $\mu_1$ for three different values of $\mu_2$.}\vspace{-.2cm}
\caption*{Constantes de stabilit\'e en fonction de $\mu_1$ pour trois valeurs de $\mu_2$.}
\vspace{-.1cm}
\end{figure}



\end{document}